\documentclass{amsart}
\usepackage{amssymb, amstext, amscd, amsmath}
\usepackage[all,cmtip]{xy}
\usepackage{enumitem}
\usepackage{xcolor}

\usepackage{verbatim}

\theoremstyle{plain}
\newtheorem{theorem}{Theorem}[section]
\newtheorem{corollary}[theorem]{Corollary}
\newtheorem{proposition}[theorem]{Proposition}
\newtheorem{lemma}[theorem]{Lemma}
\theoremstyle{definition}
\newtheorem{remark}[theorem]{Remark}
\newtheorem{example}[theorem]{Example}

\newtheorem{definition}[theorem]{Definition}

\theoremstyle{plain}
\newtheorem{question}{Question}

\newcommand{\bC}{{\mathbb{C}}}

\newcommand{\bF}{{\mathbb{F}}}

\newcommand{\bN}{{\mathbb{N}}}

\newcommand{\bR}{{\mathbb{R}}}

\newcommand{\bT}{{\mathbb{T}}}

\newcommand{\bZ}{{\mathbb{Z}}}

\newcommand{\GG}{G^{(0)}}
\newcommand{\go}{G^{(0)}}

\renewcommand{\phi}{\varphi}
\newcommand{\upchi}{{\raise.35ex\hbox{\ensuremath{\chi}}}}

\newcommand{\supp}{\operatorname{supp}}

\newcommand{\susbeteq}{\subseteq} 

\begin{document}
\title[Fourier coefficients and rapid decay]{Fourier coefficients and rapid decay in reduced groupoid C*-algebras}

\author[A.H. Fuller]{Adam H. Fuller}
\address{
Department of Mathematics\\
Ohio University\\
Athens\\
OH 45701 U.S.A.}
\email{fullera@ohio.edu}

\author[P. Karmakar]{Pradyut Karmakar}
\email{pk481519@ohio.edu}
\begin{abstract}
Let $\Sigma \rightarrow G$ be a twist over a locally compact Hausdorff \'{e}tale groupoid $G$.
Given $f$ in the reduced C$^*$-algebra $C_r^*(\Sigma;G)$ with open support $U \subseteq G$ we ask when $f$ lies in the closure of the compactly supported sections on $U$.
Suppose $G$ satisfies the rapid decay property with respect to a length function $L$.
We give a positive answer to our question in two instances: when  $L$ is conditionally negative-definite, and when $L$ is the square-root of a locally negative type function on $G$.
\end{abstract} 
\maketitle

\section{Introduction}
Let $G$ be a Hausdorff locally compact groupoid, and let $\Sigma \rightarrow G$ be a twist over $G$.
It is well known that there is a contractive homomorphism from the reduced C$^*$-algebra $C_r^*(\Sigma;G)$ and the space of $C_0$-sections on the line bundle $L$ over $G$ induced by $\Sigma$ with the $\sup$-norm.
We can thus, by an abuse of notation, view $C_r^*(\Sigma;G)$ as a subset of $C_0(\Sigma;G)$. 
Given an $f \in C_r^*(\Sigma; G)$, let $U = \{\gamma \in G \colon f(\gamma) \neq 0\}$ be the open support of $f$, as viewed as a section on $G$.
This note is motivated by the following question: does the support of $f$ determine $f$?
Explicitly, we ask:
\begin{question}\label{question}
If $f \in C_r^*(\Sigma;G)$ and $U$ is the support of $f \in G$, are there compactly supported sections on $U$ which approximate $f$ in the C$^*$-norm of $C_r^*(\Sigma; G)$? 
\end{question}
When $\gamma \in G$, we view $f(\gamma)$ as the $\gamma$-th Fourier coefficient of $f$.
Thus, we view Question~\ref{question} as asking if we can recover elements of $C_r^*(\Sigma;G)$ from their Fourier coefficients.

To motivate this question further, let $\Gamma$ be a discrete group acting on a Hausdorff locally compact space $X$, 
and $C_0(X) \rtimes_r \Gamma$ be the reduced crossed product.
If $a \in C_0(X) \rtimes_r \Gamma$, then $a$ has a unique Fourier series
$$ a \sim \sum_{g \in \Gamma} a_g \cdot g,$$
where each $a_g \in C_0(X)$.
One can alternatively see the crossed product $C_0(X) \rtimes_r \Gamma$ as a groupoid C$^*$-algebra.
That is $C_0(X) \rtimes_r \Gamma$ is isomorphic to the reduced groupoid C$^*$-algebra $C_r^*(\Gamma \times X)$ of the transformation group $\Gamma \times X$.
From this perspective, if $a \in C_r^*(\Gamma \times X)$ then its support is precisely
$$ U = \bigcup_{g \in \Gamma}[\{g\} \times \supp(a_g)],$$
where $a \sim \sum_g a_g \cdot g$ is the Fourier series of $a$.
Thus, in this setting, Question~\ref{question} is explicitly asking if $a$ can be recovered from its Fourier series.

The study of the relationship between $a \in C_0(X) \rtimes_r \Gamma$ and the Fourier series of $a$ has a long history.
Let $a \sim \sum_g a_g \cdot g$ be the Fourier series of an element $a \in C_0(X)\rtimes_r \Gamma$.
Some affirmative answers to Question~\ref{question} for crossed products (often in more generality than stated here) include, but are not limited to: when $\Gamma$ is amenable \cite{Zeller-Meier}; when $\Gamma$ is weakly amenable \cite{BedCon2015}; when $\Gamma$ has the approximation property of Haagerup-Kraus \cite{CraNeu2022, Suzuki2017}; and when $\Gamma$ has the Haagerup property and the rapid decay property with respect to the associated length function \cite{BedCon2015}.

For a Hausdorff locally compact groupoid $G$ with a twist $\Sigma \rightarrow G$, Question~\ref{question} is again known to have a positive in certain cases.
When $G$ is amenable, this follows from \cite[Theorem~4.2]{BroExeFulPitRez2021} (corrected proof in \cite{brown2024corrigendum}; see Theorem~\ref{thm: amen}).
When $f \in C_r^*(\Sigma;G)$ has closed open support, then Question~\ref{question} also has an affirmative answer by \cite[Lemma~3.8]{DGNRW} (see Theorem~\ref{thm: clopen}).
There are, however, times when there is a negative answer to Question~\ref{question}.
One class of such examples are the HLS groupoids introduced in \cite{HLS} (see Example~\ref{ex: HLS}).

In this paper, we explore Question~\ref{question} for groupoids satisfying the rapid decay property.
In our first result in this direction, Theorem~\ref{thm: rdp supp}, we show that if $G$ has the rapid decay property with respect to a continuous length function $L$, then there is a dense $*$-algebra $A \subseteq C_r^*(\Sigma;G)$ such that if $f \in A$ and $\supp(f) = U$ then $f \in \overline{C_c(\Sigma|_U;U)}^{\|\cdot\|_r}$.
The dense algebra $A$ is the Schwartz space of $G$ with respect to $L$, $H^{2,L}(\Sigma;G)$.

In Theorem~\ref{thm: negative RDP} we show that $f \in \overline{C_c(\Sigma|_U;U)}^{\|\cdot\|_r}$ for all $f \in C_r^*(\Sigma;G)$ when $L$ is conditionally negative-definite.
In Theorem~\ref{thm: main} we show that if $G$ is a second-countable locally compact Hausdorff groupoid satisfying the Haagerup property and the rapid decay property with respect an associated length function, then Question~\ref{question} has an affirmative answer.

Our three main theorems, Theorem~\ref{thm: rdp supp}, Theorem~\ref{thm: negative RDP}, and Theorem~\ref{thm: main}, give groupoid analogues of results of B\'edos and Conti for groups, crossed products, and twisted crossed products. 
In particular, Theorem~\ref{thm: rdp supp} should be compared to \cite[Theorem~1.1]{BedCon2009}; Theorem~\ref{thm: negative RDP} should be compared to \cite[Corollary~6.5]{BedCon2015}; and Theorem~\ref{thm: main} should be compared to \cite[Theorem~5.9]{BedCon2009}.
The tools we use are, however, necessarily different.
In particular, we apply results of Hou \cite{Hou2017}, Weygandt \cite{weygandt2023rapid}, and Kwa\'{s}niewski-Li-Skalski \cite{MR4404070}.
The study of rapid decay for groupoids was introduced by Hou \cite{Hou2017}, and extended to twists by Weygandt \cite{weygandt2023rapid}.
The Haagerup property for Fell bundles over groupoids was introduced by Kwa\'{s}niewski, Li and Skalski \cite{MR4404070}. We use the definitions from \cite{MR4404070} in Section~\ref{sec: haagerup}.

\section{Groupoids and their $C^{\ast}$-algebras} 
We briefly summarize the key definitions and concepts needed for groupoids, twists, and their C$^*$-algebras.
We point the reader to \cite[Part~II]{SiGaWi2020} and \cite[Section~2]{BrFuPiRe2021} for a detailed background.
A \emph{groupoid} $G$ is a small category where each morphism has inverse.
The \emph{source} map and \emph{range} map are given by 
\begin{align*}
    &s(\gamma)=\gamma^{-1} \gamma \\
    &r(\gamma)=\gamma \gamma^{-1},
\end{align*}
for all $\gamma \in G$, respectively.
A \emph{topological groupoid} is a groupoid with a topology in which multiplication and inversion maps are continuous.
An open set $B$ in a topological groupoid is called a \emph{bisection} if $s|_{B}$ is a homeomorphism.
A topological groupoid $G$ is called \emph{\'etale} if the topology consists of a basis of open bisections.
The units of a groupoid $G$ are the objects of the small category.
We denote the units of $G$ by $\GG$.
That is
$$ \GG = \{ \gamma^{-1} \gamma \colon \gamma \in G \}.$$
Note that for an \emph{\'etale Hausdorff groupoid} $G$, the unit space $\GG$ is a closed and open subset of $G$.
For every element $x$ of the unit space $\GG$, denote $G_x=\{\gamma \in G: s(\gamma)=x\}$ and $G^x=\{\gamma \in G \colon r(\gamma)=x\}$.
For $x \in \GG$, the \emph{isotropy groupoid} $\mathrm{Iso}(G)=\bigcup_{x \in \GG} G_x^{x}$, where $G_x^{x}:=G_x \cap G^x$.
A groupoid is \emph{topologically principal} if $\{x \in \GG: G_x^{x}=\{x\}\}$ is dense in $\GG$. 
A groupoid is called \emph{principal} if the set $\{\gamma \in G: s(\gamma)=r(\gamma)\}$ is $\GG$.
For an \'etale groupoid $G$, note that $G_x$ and $G^{x}$ are discrete subsets of $G$ by \cite[Corollary~8.4.10]{SiGaWi2020}.
Throughout this article, all groupoids are topological and assumed to be locally compact Hausdorff.

 Let $G$ be a locally compact Hausdorff \'{e}tale groupoid. 
 Let $\Sigma$ be a groupoid. Then $\Sigma$ is a \emph{twist} over $G$, denoted $\Sigma\rightarrow G$, if
 $$ \bT \times \GG \xrightarrow{\iota} \Sigma \xrightarrow{q} G$$
is a locally trivial central extension of $G$.
A twist $\Sigma \rightarrow G$ determines a line bundle $L$ over $G$.
Explicitly $L = (\bC \times \Sigma)/\sim$, where $(z_1,\gamma_2) \sim (z_2,\gamma_2)$ if there is a $\lambda \in \bT$ such that $(z_1,\gamma_1) = (\overline{\lambda}z_2,\lambda\cdot \gamma_2)$.
The surjection $P \colon L \rightarrow G$ is given by $P([\lambda,\gamma])=q(\gamma).$
For $[\lambda,\gamma] \in L$ we define the absolute value to be $|[\lambda,\gamma]|=|\lambda|$.

Denote by $C_c(\Sigma;G)$ the space of all compactly supported continuous sections $f \colon G \rightarrow L$.
The space $C_c(\Sigma;G)$ becomes a $*$-algebra with the convolution product
\begin{equation}\label{eq: conv product}
(f * g)(\gamma) = \sum_{\alpha\beta = \gamma} f(\alpha) g(\beta),
\end{equation}
and adjoint
\begin{equation}\label{eq: adjoint}
 f^*(\gamma) = \overline{f(\gamma^{-1})}.
\end{equation}
Note that the space $C_c(\Sigma;G)$ is sometimes defined as the space of continuous, compactly supported covariant function $f \colon \Sigma \rightarrow G$.
The equivalence of these two definitions is described in \cite[Section~2.2]{BrFuPiRe2021}.

The \emph{reduced C$^*$-algebra} $C_r^*(\Sigma;G)$ of $\Sigma\rightarrow G$ is the completion of $C_c(\Sigma;G)$ with respect to the reduced norm $\|\cdot\|_r.$
This is the norm induced on by the left regular representations of $C_c(\Sigma;G)$ on $\ell^2(\Sigma_x; G_x)$ for each $x \in \GG$, where $\ell^2(\Sigma_x; G_x)$ is the square-summable sections on $G_x$.

When $U \subseteq G$, we denote by $C_c(\Sigma|_U;U)$ the sections in $C_c(\Sigma;G)$ supported in $U$.
By \cite[Lemma~2.7]{BrFuPiRe2021} $H \subseteq G$ is an open subgroupoid, then $\Sigma|_H \rightarrow H$ is a twist, and the inclusion map $C_c(\Sigma|_H;H) \hookrightarrow C_c(\Sigma;G)$ extends to an inclusion of the C$^*$-algebras $C_r^*(\Sigma|_H;H)$ into $C_r^*(\Sigma;G).$

\subsection{Fourier coefficients}
The identity map $j(f) = f$ between $C_c(\Sigma;G)$ with the reduced norm $\|\cdot\|_r$, and $C_c(\Sigma;G)$ with the $\sup$-norm $\|\cdot\|_\infty$ is a contractive $*$-homomorphism (with the convolution product \eqref{eq: conv product} and adjoint \eqref{eq: adjoint} on each copy of $C_c(\Sigma;G)$).
The map $j$ extends to an injective contractive homomorphism $j \colon C_r^*(\Sigma;G) \rightarrow C_0(\Sigma;G)$ \cite[Proposition~2.21]{BrFuPiRe2021}.
Henceforth, we will identify $C_r^{\ast}(\Sigma; G) \subseteq C_0(\Sigma; G)$ suppressing $j$.
Given $f \in C_r^*(\Sigma;G)$ and $\gamma \in G$, we view $f(\gamma)$ as the $\gamma$-th \emph{Fourier coefficient} of $f$.
When $G$ is a group $\Gamma$ or a transformation group $\Gamma \times X$ this terminology is consistent with the usual usage of the term Fourier coefficient.

\begin{definition}
Let $G$ be a locally compact Hausdorff \'{e}tale groupoid, and let $\Sigma \rightarrow G$ be a twist. 
Given $f \in C_r^*(\Sigma;G)$ we define the \emph{support of $f$}, denoted by $\supp(f)$, to be the open set
$$ \supp(f) = \{\gamma \in G \colon f(\gamma) \neq 0 \}.$$
\end{definition}

We are interested in when an element $f \in C_r^*(\Sigma;G)$ can be ``recovered" from its Fourier coefficients.
We ask, if $U = \supp(f)$ is $f$ in $\overline{C_c(\Sigma|_U;U)}^{\|\cdot\|_r}$.
Our main results are positive answers to this question in the presence of the rapid decay property.
These are given in Theorem~\ref{thm: rdp supp}, Theorem~\ref{thm: negative RDP}, and Theorem~\ref{thm: main}.

In this subsection, we collate known positive results and counter-examples.
The following two theorems were proved, and stated, for subalgebras of $C_r^*(\Sigma;G)$ and subgroupoids of $G$. 
We state them here in more generality; the arguments in the original sources do not need to be altered significantly to do this.
The following theorem is stated in \cite{BroExeFulPitRez2021}.
We warn the reader that the proof in \cite{BroExeFulPitRez2021} is incorrect, but has been corrected in \cite{brown2024corrigendum}.

\begin{theorem}[{cf.~\cite[Theorem~4.2]{BroExeFulPitRez2021}}]\label{thm: amen}
Let $G$ be an amenable locally compact Hausdorff \'{e}tale groupoid, and let $\Sigma \rightarrow G$ be a twist.
Take $f \in C_r^*(\Sigma;G)$ with $\supp(f) = U$.
Then $f \in \overline{C_c(\Sigma|_U;U)}^{\|\cdot\|_r}$.
\end{theorem}

The next theorem gives a positive answer when we assume the support of $f$ is clopen.
In \cite{DGNRW}, this is used to identify Cartan subalgebras of $C_r^*(\Sigma;G)$.

\begin{theorem}[{cf.~\cite[Lemma~3.8]{DGNRW}}]\label{thm: clopen}
    Let $G$ be a locally compact Hausdorff \'{e}tale groupoid, and let $\Sigma \rightarrow G$ be a twist.
Take $f \in C_r^*(\Sigma;G)$ with $\supp(f) = U$.
If $U$ is closed in $G$, then $f \in \overline{C_c(\Sigma|_U;U)}^{\|\cdot\|_r}$.
\end{theorem}

It is not always the case that we can recover elements of $C_r^*(\Sigma;G)$ from their supports.
We will now describe how failure of $G$ to be inner-exact (in the sense of \cite{BonLi2020}) gives rise to such elements.

A subset $U \subseteq \GG$ is \emph{invariant} if $r(\gamma) \in U$ implies $s(\gamma) \in U$ for all $\gamma \in G$.
Given an invariant $U \subseteq \GG$ we define $G_U$ to be the subgroupoid
$$G_U = \{\gamma \in G \colon r(\gamma) \in U\}.$$
Given an open invariant set $U \subseteq \GG$, with $F = \GG \backslash U$, we get the following commutative diagram
\begin{equation*}
\label{comm diagram}
\xymatrix{
C_r^*(\Sigma_{G|_U};G_U)\ar[r]\ar[d]^{E_U} & C_r^*(\Sigma;G)\ar[d]^{E}\ar@{->>}[r]^\phi &C_r^*(\Sigma|_{G_F};G_F)\ar[d]^{E_F}\\
 C_0(U)\ar[r] & C_0(\go)\ar@{->>}[r]^\theta& C_0(F)
}
\end{equation*}
where the vertical arrows are given by faithful conditional expectations.
The bottom line will be exact, however the top line may not be exact.
The ideal $\ker\phi$ will consist of all $f \in C_r^*(\Sigma;G)$ such that $\supp(f) \subseteq G_U$.
Thus, if the top line is not exact, then there exist $f \in C_r^*(\Sigma;G)$ such that $\supp(f) \subseteq G_U$ but
$$f \notin \overline{C_c(\Sigma|_{G_U};G_U)}^{\|\cdot\|_r} = C_r^*(\Sigma|_{G_U};G_U).$$
Higson-Lafforgue-Skandalis \cite{HLS} give a class of groupoids where the top line will not be exact.

\begin{example}\label{ex: HLS}
Let $\Gamma$ be an infinite group and let $\{K_n\}_n$ be a nested family of normal finite-index subgroups of $\Gamma$ such that $K_{n+1} \subseteq K_n$, and $\bigcap_n K_n = \{e\}.$ 
For each $n$, let $\Gamma_n = \Gamma/K_n$, and let $\Gamma_\infty = \Gamma$.
The \emph{HLS groupoid} for $\Gamma,\ \{K_n\}_n$ is defined as the disjoint union
$$ G = \bigsqcup_{n \in \bN \cup \{\infty\}} \Gamma_n$$
with topology induced by the compact topology on $\bN \cup \infty$; see \cite{anantharamandelaroche2021exactgroupoids}, \cite{HLS} or\cite{Wil2015} for details.
Consider the sequence
\begin{align}\label{eq: HLS inexact}
C_r^*(G_\bN) \rightarrow C_r^*(G) \rightarrow C_r^*(\Gamma).
\end{align}
Higson-Lafforgue-Skandalis \cite{HLS} give sufficent conditions for \eqref{eq: HLS inexact} not to be in exact.
In fact, \eqref{eq: HLS inexact} is exact if and only if $\Gamma$ is amenable \cite{anantharamandelaroche2021exactgroupoids, Wil2015}.

Thus, if $G$ is an HLS-groupoid constructed from a non-amenable group $\Gamma$, there exist $f \in C_r^*(\Gamma)$ with $\supp(f) \subseteq G_\bN$ but $f \notin C_r^*(G_\bN)$.
Willett \cite{Wil2015} gives an example where $\Gamma = \bF_2$ and $C_r^*(G)/C_r^*(G_\bN) \simeq C^*(\bF_2)$, the full C$^*$-algebra for $\bF_2$.
\end{example}

\section{Schwartz space and Rapid Decay Property}\label{sec: schwartz}

\subsection{The $2$-norm}
In addition to the reduced norm $\|\cdot\|_r$ and the $\sup$-norm $\|\cdot\|_\infty$, we will need several more norms on $C_c(\Sigma; G)$, for a twist $\Sigma\rightarrow G$.
We recall here the  $2$-norm and some key facts.

\begin{definition}
Define a norm $\|\cdot\|_{2}$ on $C_c(\Sigma; G)$ by
\begin{align*}
    \|f\|_{2} &= \sup_{x\in G^{(0)}} \max \left\{ \sum_{\gamma \in G_x} |f(\gamma)|^2,\ \sum_{\gamma \in G^x} |f(\gamma)|^2 \right\}^{1/2},
\end{align*}
for each $f \in C_c(\Sigma; G)$.
Denote 
by $\ell^2(\Sigma; G)$ the space $\overline{C_c(\Sigma; G)}^{\|\cdot\|_{2}}$.
\end{definition}

We note that for any  $g \in \ell^2(\Sigma;G)$ he have that
\begin{align*}
    \|g\|_{2} &= \sup_{x\in G^{(0)}} \max \left\{ \sum_{\gamma \in G_x} |g(\gamma)|^2,\ \sum_{\gamma \in G^x} |g(\gamma)|^2 \right\}^{1/2}.
\end{align*}

The  norms on $C_c(\Sigma;G)$, $\|\cdot\|_\infty$, $\|\cdot\|_r$, and $\|\cdot\|_{2}$ are related by the following inequality
$$ \|f\|_\infty \leq \|f\|_{2} \leq \|f\|_r $$
for each $f \in C_c(\Sigma;G)$, see e.g. \cite{RenBook}.

Closely related to the $2$-norm is the function
\begin{align}
    \label{eq: 2 func} x & \mapsto \max \left\{ \sum_{\gamma \in G_x} |f(\gamma)|^2,\ \sum_{\gamma \in G^x} |f(\gamma)|^2\right\}^{1/2}
\end{align}
for $x \in \go$ and a fixed $f \in C_c(\Sigma;G)$.
This function is continuous on $\go$, since the counting measure of $G$ forms a Haar system, see e.g. \cite[Lemma~I.2.7]{RenBook}.
Further,  $\eqref{eq: 2 func}$ is continuous for any $f \in \ell^2(\Sigma; G)$.
We prove this is in the following lemma.
That the analogous statement for the $1$-norm is continuous is \cite[Lemma~2.6]{AusOrt2022}; or one can follow a similar proof as below.
 
 \begin{lemma}\label{lem: cont 2 sum}
    Let $\Sigma \rightarrow G$ be a twist over an \'etale locally compact Hausdorff groupoid.
    Then for each $f \in \ell^2(\Sigma; G)$ the function on $\go$ defined by
    \begin{align*}
    x & \mapsto \max \left\{ \sum_{\gamma \in G_x} |f(\gamma)|^2,\ \sum_{\gamma \in G^x} |f(\gamma)|^2\right\}^{1/2}
    \end{align*}
    is continuous.
\end{lemma}

\begin{proof}
For $f \in \ell^2(\Sigma; G)$ let $T_f$ be the function
\begin{align*}
 T_f \colon \go &\longrightarrow \bC \\ 
 x &\longmapsto \sum_{\gamma \in G_x} |f(\gamma)|^2.
\end{align*}
We will show that $T_f$ is continuous.
A similar argument works for the function $x \mapsto \sum_{\gamma \in G^x} |f(\gamma)|^2$, and the result readily follows from there.

We have already noted that $T_g$ is continuous for each $g \in C_c(\Sigma;G)$.
Now take an arbitrary $f \in \ell^2(\Sigma; G)$.
We claim that for each $\varepsilon>0$ there exists a compact set $K_{\epsilon} \subseteq G$ such that for each $x \in \GG$, we have
    \begin{align*}
        \sum_{\gamma \in G_x\backslash K_{\varepsilon}}|f(\gamma)|^{2} < \varepsilon.
    \end{align*}
    Indeed, we can choose $g \in C_c(\Sigma; G)$ such that $\|f-g\|_{2}^{2}< \varepsilon$. 
    Let $K\subset G$ be a compact set such that $\supp(g) \subset K$.
    Then the claim holds, with $K_\varepsilon = K.$

    Fix $n$, choose $K = K_{\frac{1}{n}}$.
    Let $U$ be an open set in $G$ such that $K \subseteq U$ and $\bar{U}$ is compact in $G$.
    By Urysohn's lemma, there exists a continuous function {$h_n \colon G \rightarrow [0,1]$} such that $h_n(\gamma)=1$ if $\gamma \in K$ and $\text{supp}(h_n) \subseteq U$.
    Let $g_n = fh_n$. Then $g_n \in C_c(\Sigma; G)$.

    One can show that the functions $(T_{g_n})_n$ coverge uniformly to $T_f$.
    As each $T_{g_n}$ is continuous, it follows that $T_f$ is continuous.
\end{proof}

We now show that functions in $\ell^2(\Sigma; G)$ are determined by their supports.

\begin{proposition}\label{prop: ell2 support}
        Let $G$ be an \'{e}tale, locally compact Hausdorff groupoid and let $\Sigma \rightarrow G$ be a twist.
    If $f \in \ell^2(\Sigma; G)$ and $\supp(f) \subseteq U$ for some open $U \subseteq G$, then $f \in \overline{C_c(\Sigma|_U;U)}^{\|\cdot\|_{2}}$.
\end{proposition}

\begin{proof}
Take $f \in \ell^2(\Sigma; G)$ with $\supp(f) \subseteq U$ for an open set $U \susbeteq G$.
Take any $\varepsilon > 0$. 
We will show that there exists $g \in C_c(\Sigma_U; U)$ so that $\|f - g\|_2 < \varepsilon$.

Take $K \subseteq G$ compact such that for all $x \in \go$
$$ \sum_{\gamma \in G_x\backslash K} |f(\gamma)|^2<\varepsilon,\text{ and } \sum_{\gamma \in G^x\backslash K} |f(\gamma)|^2< \varepsilon.$$
Note, $K$ exists since $\ell^2(\Sigma; G) = \overline{C_c(\Sigma; G)}^{\|\cdot\|_2}$.

Fix $x \in s(K)$.
There exist finite sets $F_x \subseteq G_x$ such that
$$ \sum_{\gamma \in G_x\backslash F_x} |f(\gamma)|^2<\varepsilon/3.$$
Let $F_x = \{ \gamma_1, \ldots, \gamma_n\}.$
For each $\gamma_i \in F_x$ there are bisections $V_i$ and $U_i$ such that $\gamma_i \in V_i \subseteq \overline{V_i} \susbeteq U_i \subseteq U$, where $\overline{V_i}$ is compact.
As $f$ is continuous, we can further assume that for each $\gamma_i \in F_x$ and $\sigma \in U_i$ that
$$ | |f(\gamma_i)|^2 - |f(\sigma)|^2| < \frac{\varepsilon}{3n},$$
by shrinking $U_i$ is necessary.

The function
$$ T_f(x) = \sum_{\gamma \in G_x} |f(\gamma)|^2 $$
is continuous by Lemma~\ref{lem: cont 2 sum}. 
Hence, there is an open neighborhood $A_x$ of $x$ such that for each $y \in A_x$
$$ |T_f(x) - T_f(y)| < \varepsilon/3.$$
Let 
$W_x = A_x \cap \bigcap_{i= 1}^ns(V_i)$
and let
$ K_x = \bigcup_{i=1}^n \overline{V_i}.$

Take $y \in W_x$.
For each $1\leq i \leq n$, let $\sigma_i = G_y\cap V_i$, and let $F = \{\sigma_1,\ldots,\sigma_n\}$.
Then
\begin{align*}
\sum_{\sigma \in G_y \backslash F} |f(\sigma)|^2 
&=  \left| T_f(y) - \sum_{i=1}^n|f(\sigma_i)|^2 \right| \\
&= \left| T_f(y) - T_f(x) + T_f(x) - \sum_{i=1}^n|f(\sigma_i)|^2 \right| \\
&\leq  \left| T_f(y) - T_f(x)\right| + \left|\sum_{\gamma \in G_x\backslash F_x}|f(\gamma)|^2\right| + \left| \sum_{i=1}^n(|f(\gamma_i)|^2-|f(\sigma_i)|^2) \right| \\
&< \frac{\varepsilon}{3} + \frac{\varepsilon}{3} + n\frac{\varepsilon}{3n}= \varepsilon.
\end{align*}
Note that $F = \{\sigma_1, \ldots, \sigma_n\} \subseteq K_x$.
Thus, if $y \in W_x$, then
$$ \sum_{\sigma \in G_y \backslash K_x} |f(\sigma)|^2 < \varepsilon.$$

As $s(K)$ is compact, there are $x_1, \ldots, x_m$ and associated sets constructed as above $W_1, \ldots, W_m$ covering $s(K)$. 
Let $L' = \bigcup_{i=1}^m K_i$, where $K_i=K_{x_i}$ is the compact set constructed as above.

Applying the same argument using the range map instead of source, we get open sets $W^1,\ldots,W^k$,  covering $r(K)$ and a compact sets $K^1,\ldots,K^k$ such that if $x \in W^i$
$$ \sum_{\gamma \in G^x\backslash K^i} |f(\gamma)|^2< \varepsilon.$$
Let $L'' = \bigcup_{i=1}^k K^i$.

Consider the compact set $L = L'\cup L'' \subseteq U$.
By Urysohn's Lemma there is a $h \in C_c(U)$ such that $h(U) \subseteq [0,1]$ and $h(\gamma) = 1$ for all $\gamma \in L$.
Let $g = hf$.
Then $g \in C_c(\Sigma_U; U)$.
Take any $x\in \go$.
If $x\in s(K)$, then $x \in W_i$ for some $i$ and so
\begin{align*}
\sum_{\gamma\in G_x}|f(\gamma) - g(\gamma)|^2 
&= \sum_{\gamma\in G_x\backslash L}|f(\gamma) - g(\gamma)|^2 
\leq \sum_{\gamma\in G_x\backslash L}|f(\gamma)|^2 \\
&\leq \sum_{\gamma\in G_x\backslash K_i}|f(\gamma)|^2 
< \varepsilon.
\end{align*}
If $x \notin s(K)$, then $G_x \cap K = \emptyset$, and so
\begin{align*}
\sum_{\gamma\in G_x}|f(\gamma) - g(\gamma)|^2 
&\leq \sum_{\gamma\in G_x}|f(\gamma)|^2 
= \sum_{\gamma\in G_x\backslash K}|f(\gamma)|^2 
< \varepsilon.
\end{align*}
One can similarly show that for all $x \in \go$
$$ \sum_{\gamma\in G^x}|f(\gamma) - g(\gamma)|^2 < \varepsilon. $$
Thus $\|f-g\|_2^2 <\varepsilon$ and so $f \in \overline{C_c(\Sigma|_U; U)}^{\|\cdot\|_2}$.
\end{proof}

\subsection{Length functions and the Schwartz space}
We will now recall the definition of the Schwartz space of $\Sigma \rightarrow G$ with respect to a length function $L$ and record some key properties.

\begin{definition}[{\cite{Hou2017}}]\label{lnfn}
Let $G$ be an \'{e}tale groupoid.
    A function $L \colon G \rightarrow [0, \infty)$ is a \emph{length function on G} if
    \begin{enumerate}
        \item $L(x)=0$ for all $x \in \GG$;
        \item $L(\gamma)=L(\gamma^{-1})$ for all $\gamma \in G$, i.e. $L$ is a symmetric function; and
        \item $L(\gamma_1 \gamma_2) \leq L(\gamma_1) + L(\gamma_2)$ for all $(\gamma_1,\gamma_2) \in G^{(2)}$, i.e. $L$ is subadditive.
    \end{enumerate}
\end{definition}

The following definition is adapted from \cite[Definition~2.9]{weygandt2023rapid}.

\begin{definition}\label{def: schwartz}
Let $\Sigma \rightarrow G$ be a twist, where $G$ is a locally compact \'{e}tale groupoid.
Let $L$ be a length function on $G$.
For each $p$ in the non-negative integers $\bZ_+$, define norms on $C_c(\Sigma; G)$ by
\begin{align*}
    &\|f\|_{2, p, s, L}= \sup_{x \in \GG}\big\{ \big(\sum_{\gamma  \in G_x} |f(\gamma)|^{2}(1+ L(\gamma))^{2p}\big)^{\frac{1}{2}}\big\},\\
    &\|f\|_{2. p, r, L}= \sup_{x \in \GG}\big\{ \big(\sum_{\gamma  \in G^x} |f(\gamma)|^{2}(1+ L(\gamma))^{2p}\big)^{\frac{1}{2}}\big\}, \text{ and}\\
& \|f\|_{2, p, L}= \max\{\|f\|_{2, p, s, L}, \|f\|_{2, p, r, L}\}\,.
\end{align*}
For each $p$ in $\bZ_+$, let $H^{2,L,p}(\Sigma; G)$ be the completion of $C_c(\Sigma; G)$ with respect to the norm $\|\cdot\|_{2, p, L}$.
We define the \emph{Schwartz space} $H^{2, L}(\Sigma; G)$ by
$$ H^{2,L}(\Sigma; G) = \bigcap_{p\in \bZ_+} H^{2,L,p}(\Sigma; G) \cap C_0(\Sigma; G).$$
That is, $H^{2,L}(\Sigma; G)$ is the completion of $C_c(\Sigma; G)$ by the norms $\|\cdot\|_{2, p, L}$ and $\|\cdot\|_{\infty}$ for all $p \in \bZ_+$.
\end{definition}

\begin{remark}
In \cite{weygandt2023rapid} the space $C_c(\Sigma; G)$ is defined to be the space of compactly supported, covariant functions $f \colon \Sigma \rightarrow G$.
By \cite[Section~2]{BrFuPiRe2021} that description is equivalent to the definition of $C_c(\Sigma; G)$ in terms of sections on the line bundle $L$ over $G$ given in this paper.
In \cite[Definition~2.9]{weygandt2023rapid}, the Schwartz space is thus described as a completion of functions on $\Sigma$, not sections on $G$ as we have done in Definition~\ref{def: schwartz}.
The two definitions are, however, equivalent.
\end{remark}

\begin{definition}
Let $\Sigma \rightarrow G$ be a twist, where $G$ is a locally compact \'{e}tale groupoid.
Let $L$ be a length function on $G$.
Let $U \subseteq G$. 
Define the \emph{Schwartz space over $U$} as
$$ H^{2,L}(\Sigma|_U;U) = \bigcap_{p\geq 1} \overline{C_c(\Sigma|_U;U)}^{\|\cdot\|_{2,p,L}} \cap C_0(\Sigma|_U; U).$$
\end{definition}

In keeping with our theme, we present the following corollary to Proposition~\ref{prop: ell2 support}, stating that elements of the Schwartz space in the closure of the compactly supported sections on their support.

\begin{proposition}\label{prop: schwartz supp}
Let $\Sigma \rightarrow G$ be a twist, where $G$ is a locally compact \'{e}tale groupoid. Let $L \colon G \rightarrow \bC$ be a continuous length function on $G$.

If $f \in H^{2,L,p}(\Sigma; G)$ with open support $U \subseteq G$,
then $f \in \overline{C_c(\Sigma|_U;U)}^{\|\cdot\|_{2,p,L}}$.
In particular, if $f \in H^{2,L}(\Sigma; G)$ with open support $U \subseteq G$,
then $f \in H^{2,L}(\Sigma|_U;U)$. 
\end{proposition}

\begin{proof}
The result follows immediately from Proposition~\ref{prop: ell2 support} on noting that if $f \in H^{2,L,p}(\Sigma;G)$ then $f(1+L)^p \in \ell^2(\Sigma;G)$.
\end{proof}

\begin{definition}[{\cite[Definition~2.10]{weygandt2023rapid}}]
    Let $G$ be an \'etale groupoid and let $L \colon G \rightarrow [0,\infty)$ be a length function on $G$.
    Let $\Sigma \rightarrow G$ be a twist over $G$.
    The twist $\Sigma \rightarrow G$ is said to have the \emph{rapid decay property} with respect to a length function $L$ if there exists a constant $c > 0$ and a non-zero integer $p \in \bZ_+$ such that
\begin{align}
\label{rdpdef}
    \|f\|_{r} \leq c \|f\|_{2, p, L}&\text{ for all $f \in C_c(\Sigma; G)$}\,.
    \end{align}
\end{definition}

In \cite[Proposition~2.12]{weygandt2023rapid} it is shown that the rapid decay property does not depend on the specific twist $\Sigma$ on $G$.
That is, $\Sigma \rightarrow G$ has the rapid decay property with respect to a length function $L$, if and only if $G$ has the rapid decay property with respect to $L$ in the sense of \cite[Definition~3.2]{Hou2017}.

By \cite[Proposition~2.13]{weygandt2023rapid}, if a twist $\Sigma \rightarrow G$ has the rapid decay property with respect to a length function $L$ then $H^{2,L}(\Sigma; G) \subseteq C_r^{\ast}(\Sigma; G)$.
Further, if $L$ is also continuous, then the Schwartz space $H^{2, L}(\Sigma; G)$ is a dense Fr\'echet $*$-subalgebra of the reduced groupoid $C^{\ast}$- algebra $C_r^{\ast}(\Sigma; G)$.
In this case, by minor adjustments to \cite[Proposition~3.4]{Hou2017}, $G$ having the rapid decay property with respect to $L$ is equivalent to there being a $p \in \bZ_+$ and a constant $c>0$ such that
$$ \|f\|_{r} \leq c \|f\|_{2, p, L}$$
for all $f$ in $H^{2,L,p}(\Sigma; G)$.

We can now present the main theorem of this section, which says that if $G$ has the rapid decay property with respect to some length function, then there is a dense subalgebra of $C_r^*(\Sigma;G)$ for which there is an affirmative answer to Question~\ref{question}.
That is, for a dense subalgebra we can recover elements from the compactly supported sections on their supports.
The dense algebra in question is, of course, the Schwartz space $H^{2,L}(\Sigma;G)$.
This Theorem is a generalization of \cite[Theorem~1.1]{BedCon2009} for groups with the rapid decay property.

\begin{theorem}\label{thm: rdp supp}
Let $G$ be a locally compact Hausdorff \'{e}tale groupoid and let $\Sigma \rightarrow G$ be a twist of $G$.
Suppose $L$ is a continuous length function on $G$ for which $G$ has the rapid decay property.
If $f \in H^{2,L}(\Sigma;G) \subseteq C_r^*(\Sigma;G)$ and $\supp(f) = U$, then
$f \in \overline{C_c(\Sigma|_U;U)}^{\|\cdot\|_r}$.
\end{theorem}

\begin{proof}
    As $G$ has the rapid decay property with respect to $L$ we that $H^{2,L}(\Sigma;G) \subseteq C_r^*(\Sigma;G)$.
    Further, there exists $p \in \bZ_+$ and a constant $c>0$ such that $\|g\|_r \leq c\|g\|_{2,p,L}$ for all $g \in C_c(\Sigma;G)$.

    Take any $f \in H^{2,L}(\Sigma;G)$ and let $U = \supp(f)$ be its open support.
    Fix $\varepsilon > 0$.
    By Proposition~\ref{prop: schwartz supp}, there exists $g \in C_c(\Sigma|_U;U)$ such that 
    $$ \|f -g\|_{2,p,L} < \frac{\varepsilon}{c}.$$
    Hence $\|f-g\|_r < \varepsilon$, completing the result.
\end{proof}

\subsection{Conditionally negative definite length functions}
We wish to strengthen Theorem~\ref{thm: rdp supp} from the dense subalgebra $H^{2,L}(\Sigma;G) \subseteq C_r^*(\Sigma;G)$ to all of $C_r^*(\Sigma;G)$.
To do this we will need extra conditions on the length function $L$.
We will do this in two (related) cases, in Theorem~\ref{thm: negative RDP} and Theorem~\ref{thm: main}.
We first discuss a class of multipliers on $C_r^*(\Sigma;G)$.

Let $\Sigma \rightarrow G$ be a twist with $G$ a locally compact Hausdorff \'etale groupoid.
We denote the set of compactly supported complex valued continuous functions and bounded complex valued continuous functions by $C_c(G)$ and $C_b(G)$, respectively. 
Henceforth, given $k \in C(G)$, $f \in C_c(\Sigma; G)$, we define
$kf \in C_c(\Sigma; G)$ to be the point-wise product.
That is, for all $\gamma \in G$,
\begin{align*}
kf(\gamma):=k(\gamma)f(\gamma).
\end{align*} 

\begin{definition}\label{pd}
A continuous function $k \colon G\rightarrow \bC$ is called a \emph{positive-definite} function on $G$ if for any $x$ in $\GG$ and a finite set $F \subset G_x$, the matrix 
\begin{align*}
    [k(\gamma \eta^{-1})]_{\eta, \gamma \in F} \in M_{F}(\bC)
\end{align*}
is positive-definite.
\end{definition}

Versions of the following result can be found in varying amounts of generality, e.g. in \cite[Proposition~5.6.16]{BroOza2008} and \cite[Lemma~4.2]{Tak}.
The strongest and most general version can be found in \cite[Proposition~3.6]{MR4404070}.

\begin{proposition}[{\cite[Proposition~3.6]{MR4404070}}]\label{cpm}
    If $k \in C(G)$ is a positive-definite bounded complex valued function, 
    then $M_{k} \colon C_c(\Sigma; G) \ni f \mapsto kf \in C_c(\Sigma; G)$ extends to a completely positive map on $C_r^{\ast}(\Sigma; G)$ and 
     \begin{align*}
        \|M_k\| = \sup _{x \in \GG} |k(x)|=\sup_{g \in G}|k(\gamma)|.
    \end{align*}
\end{proposition}

\begin{theorem}\label{thm: mult approx}
Let $G$ be a locally compact Hausdorff \'{e}tale groupoid, and let $\Sigma \rightarrow G$ be a twist.
Suppose $(k_i)_{i \in I}$ is a uniformly bounded net of positive-definite complex valued continuous functions on $G$ that converges uniformly to $1$ on compact subsets of $G$.
If $f \in C_r^{\ast}(\Sigma; G)$, then $\|M_{k_i}(f)-f\|_{r} \rightarrow 0$.
\end{theorem}
\begin{proof}
Fix $\varepsilon >0$. 
Let $M$ be a uniform bound for the net $(k_i)_{i\in I}$.
That is, $\sup_{i \in I} \|k_i\|_{\infty} \leq M$.
By the definition of $C_r^{\ast}(\Sigma; G)$, there exists a sequence $(f_n)_{n \in \bN} \in C_c(\Sigma; G)$ such that $\|f_n-f\|_r \to 0$ as $n \to \infty$.
Thus there exists $N \in \bN$ such that for all $n \geq N$, we have
\begin{align*}
    \|f_n-f\|_{r} < \frac{\varepsilon}{3M}\,.
\end{align*}
By the triangle inequality, we obtain
\begin{align*}
    \|M_{k_i}(f)-f\|_{r} \leq \|M_{k_i}(f-f_{N})\|_{r}+\|M_{k_i}(f_{N})-f_{N}\|_{r}+\|f_{N}-f\|_{r}\,.
\end{align*}
Since $\|M_{k_i}\| = \|k_i\|_\infty$ by Proposition~\ref{cpm}, we have
\begin{align*}
    \|M_{k_i}(f-f_{N})\|_{r} &\leq  \|M_{k_{i}}\|\|f-f_{N}\|_{r}\\
    &\leq M\|f-f_{N}\|_{r} 
    < \frac{\varepsilon}{3}.
\end{align*}
By \cite[Lemma~4.5]{MR4404070}, $\lim_{i \in I}\|M_{k_i}(g)-g\|_{r}=0$ for all $g \in C_c(\Sigma;G)$.
Hence, there exists $i_0 \in I$ such that 
\begin{align*}
    \|M_{k_i}(f_{N})-f_{N}\|_{r} < \frac{\varepsilon}{3} &\text{ for all } i \geq i_0\,.
\end{align*}
Hence, for all $i \geq i_0$ we have
\begin{align*}
    \|M_{k_i}(f)-f\|_{r} < \varepsilon,
\end{align*}
proving the result.
\end{proof}

Recall that a function $\psi\colon G \rightarrow \bC$ is \emph{conditionally negative definite} if for each $x \in \GG$ and $\gamma_1,\ldots,\gamma_n \in G_x$ given $\lambda_1,\ldots,\lambda_n \in \bR$ such that $\sum_i \lambda_i =0$ we have that
$$ \sum_{i,j=1}^n\lambda_i\lambda_j \psi(\gamma_i\gamma_j^{-1})\leq 0.$$
The following theorem gives a groupoid analogue of \cite[Corollary~6.5]{BedCon2015}.

\begin{theorem}\label{thm: negative RDP}
    Let $G$ be locally compact Hausdorff \'{e}tale groupoid, and let $\Sigma \rightarrow G$ be a twist.
    Suppose there is a conditionally negative definite length function $L$ on $G$ such that $G$ has the rapid decay property with respect to $L$.
    
    Take any $f \in C_r^*(\Sigma; G)$ and let $U$ be the open support of $f$.
    Then $f \in \overline{C_c(\Sigma|_U;U)}^{\|\cdot\|_r}$.
\end{theorem}
\begin{proof}
    Define a sequence of functions $k_n(\gamma)=e^{-\frac{L(\gamma)}{n}} \in C(G)$ for all $n \in \bN$. 
    Each $k_n$ is positive definite since $L$ is a conditionally negative definite, by e.g. \cite[Chapter~3, Theorem~2.2]{BerChrRes1984}.  Further, the sequence $(k_n)_n$ converges uniformly to $1$ on compact subsets of $G$, (see the proof of \cite[Proposition~5.4]{MR4404070}). 
    Moreover, as $e^{-x}(1+x)^{s} \rightarrow 0$ as $x \rightarrow \infty$ for any positive $s \in \mathbb{R}$, we have
    \begin{align*}
    \sup_{\gamma \in G} |k_n(\gamma)|(1+L(\gamma))^{p}:=B_n < \infty\,.
    \end{align*}
    Let $f \in C_r^{\ast}(\Sigma; G)$, and let $U = \supp(f)$.
    Fix any $p \in \bN$. 
    Then
    \begin{align*}
   \sup_{x \in \GG} \sum_{{\gamma} \in G_x} |f({\gamma})|^{2}|k_n({\gamma})|^{2}(1+L({\gamma}))^{2p}
    &\leq B_n^2 \sup_{x \in \GG} \sum_{{\gamma} \in G_x} |f({\gamma})|^2\\
    &\leq B_n^2 {\|f\|_{2}^{2}}
    \leq B_n^{2} \|f\|_r^2.
    \end{align*}
    Similarly, 
    \begin{align*}
         \sup_{x \in \GG} \sum_{{\gamma} \in G_x} |f({\gamma})|^{2}|k_n({\gamma})|^{2}(1+L({\gamma}))^{2p} \leq B_n^2 \|f\|_r^2.
    \end{align*}
    Hence, for each $n$,
    \begin{align}\label{normboundinggn}
        \|M_{k_n}(f)\|_{2,p,L} \leq B_n \|f\|_r < \infty\,.
    \end{align}
  Note also that $\|M_{k_n}(f)\|_{\infty} \leq \|f\|_{\infty}$.    
  Let $C_c(\Sigma; G) \ni \varphi_k \rightarrow f$ as $k \rightarrow \infty$ in the reduced norm $\|\cdot\|_r$. 
  By \eqref{normboundinggn}, for each $n$,
     \begin{align*}\label{mltapxeqn}
      & \|M_{k_n}(\varphi_k)-M_{k_n}(f)\|_{2,p.L} \leq B_n\|\varphi_k-f\|_{r}\,.
  \end{align*}
  Consequently, we have $M_{k_n}(f) \in H^{2,L}(\Sigma; G)$ for all $n$.

  As $G$ has the rapid decay property with respect to $L$, $H^{2,L}(\Sigma; G) \subseteq C_r^{\ast}(\Sigma; G)$, and hence $M_{k_n}(f) \in C_r^{\ast}(\Sigma; G)$.
  Further $\supp(M_{k_n}f) \subseteq \supp(f) = U$.

  By Theorem~\ref{thm: mult approx}, we know $M_{k_n}(f) \rightarrow f$ in $C_r^{\ast}(\Sigma; G)$ as $n \to \infty$.
    Fix $\varepsilon > 0$. 
    There exists $n$ such that 
    $$\|M_{k_n}(f) - f \|_r < \frac{\varepsilon}{2}.$$
  By Theorem~\ref{thm: rdp supp}, there exists $g \in C_c(\Sigma|_U;U)$ such that
  $$\|M_{k_n}(f) - g \|_r < \frac{\varepsilon}{2}. $$
  Hence,
  \begin{align*}
  \|f-g\|_r &\leq \|f - M_{k_n}(f)\|_r + \|M_{k_n}(f) - g\|_r 
  < \varepsilon.
  \end{align*}
  Thus $f \in \overline{C_c(\Sigma|_U;U)}^{\|\cdot\|_r}$.
\end{proof}

Many groups satisfy the conditions of Theorem~\ref{thm: negative RDP}, including free groups, and finitely generated Coxeter groups \cite[Theorem~5]{NibRee2003}.
We refer the reader to the final paragraph of \cite{BroNib2004} for a longer list of such groups.

In \cite[Section~4]{weygandt2023rapid}, several permanence conditions are given for the rapid decay property.
We highlight one such condition here.
Let $G$ be an \'{e}tale Hausdorff groupoid satisfying the rapid decay property with respect to a length function $L$, and let $H$ be any compact groupoid.
It is shown in \cite[Proposition~4.2]{weygandt2023rapid} that the groupoid $G\times H$ has the rapid decay property with respect to the length function $\hat{L}$, defined by $\hat{L}(\gamma,\eta)=L(\gamma)$.
It is not hard to see that if $L$ is conditionally negative definite, then so is $\hat{L}$.
Thus, if $G$ satisfies the conditions of Theorem~\ref{thm: main}, then so will $G \times H$.

\section{The Haagerup property and a related length function}\label{sec: haagerup}
The Haagerup property for locally compact Hausdorff \'{e}tale groupoids was introduced in \cite{MR4404070}.
This is used in \cite{MR4404070} to study C$^*$-algebras of Fell bundles over groupoids satisfying the Haagerup property. 
We will show that (a stronger condition than) the Haagerup approximation property of \cite{MR4404070} can be used to construct a length function.
Our main theorem, Theorem~\ref{thm: main}, involves twists which have the rapid decay property with respect to this length function.

\begin{definition}[{\cite[Proposition~5.4]{MR4404070}}]\label{ndcf}
    A \emph{locally proper negative type} function on an \'etale locally compact Hausdorff groupoid $G$ is a continuous function $\psi: G \rightarrow \mathbb{R}$ satisfying the following:
    \begin{enumerate}
        \item $\psi$ is \emph{normalized}, i.e. $\psi|_{\GG}=0$;
        \item $\psi$ is \emph{symmetric}, i.e. $\psi(\gamma)=\psi(g^{-1})$  for all $g \in G$;
        \item $\psi$ is conditionally negative definite;
         \item $\psi$ is \emph{locally proper}, i.e. the function $(\psi, r, s) \colon G \rightarrow \mathbb{R} \times \GG \times \GG$ is proper\,.
    \end{enumerate}
\end{definition}

\begin{definition}[{\cite[Definition~5.6]{MR4404070}}]\label{HAP}
Let $G$ be a \emph{locally compact \'etale} groupoid. 
We say the groupoid $G$ has the \emph{Haagerup property} if there exists a net $(k_i)_{i \in I}$ of positive-definite functions on $G$ such that
\begin{enumerate}
    \item each $k_i$ is normalized, i.e., $k_i|_{\GG}=1$;
    \item each $k_i$ is a locally $C_0$-function, i.e. for each compact set $K \subset \GG$, we have $k_i|_{G_{K}^{K}} \in C_0(G_{K}^{K})$ where $G_{K}^{K}:=r^{-1}(K) \cap s^{-1}(K)$\;
    \item the net $(k_i)_{i \in I}$ converges to $1$ uniformly on compact subsets of $G$\,.
\end{enumerate}
We say that $G$ has the \emph{sequential Haagerup property} if the net $(k_i)_{i \in I}$ can be replaced by a sequence.
\end{definition}

\begin{remark}
In \cite[Definition~5.6]{MR4404070}, the Haagerup property is defined for Fell bundles over \'{e}tale groupoids. 
Though we are primary focus of study is line bundles over groupoids, we only present the definition of the Haagerup approximation property for groupoids (without any bundle structure).
This is sufficient by \cite[Lemma~5.10]{MR4404070}: a twist $\Sigma \rightarrow G$ satisfies \cite[Definition~5.6]{MR4404070} if and only if $G$ satisfies Definition~\ref{HAP}.
\end{remark}

Definition~\ref{ndcf} and Definition~\ref{HAP} are related by the following Proposition.

\begin{proposition}[{cf.~\cite[Proposition~5.4]{MR4404070}}]\label{prop: HAP equivalence}
Let $G$ be an \'{e}tale locally compact Hausdorff groupoid.
Consider the following properties
\begin{enumerate}
\item there is a non-zero locally proper negative type function on $G$;
\item $G$ has the sequential Haagerup property.
\end{enumerate}
Then (i) implies (ii).
Further, if $G$ is $\sigma$-compact then (i) and (ii) are equivalent.
\end{proposition}

\begin{corollary}
Let $G$ be an \'{e}tale locally compact second-countable Hausdorff groupoid.
Then $G$ admits a non-zero locally proper negative type function if and only if $G$ has the sequential Haagerup property.
\end{corollary}

We will show in Theorem~\ref{thm: neg to length} that a locally proper negative type function on a Hausdorff \'{e}tale groupoid $G$ induces a length function on the groupoid $G$.
When $G$ is a discrete group this is \cite[Proposition~2.25]{BedCon2009}.
As is done in \cite{BedCon2009}, we adapt arguments from \cite{BerChrRes1984}.
We include proofs here for completeness.
In the case when $G$ is a discrete group, the following lemma follows from \cite[Chapter~3, Lemma~2.1]{BerChrRes1984}.
The proof of  Theorem~\ref{thm: neg to length} adapts the proof of \cite[Chapter~4, Proposition~3.3]{BerChrRes1984} to the \'etale groupoid setting.

\begin{lemma} \label{ndl}
    Let $G$ be an \'etale Hausdorff groupoid.
    Fix any $x \in \GG$.
    Let $\psi\colon G \rightarrow \mathbb{R}$ be a locally proper negative type function.
    Fix any $x \in G^{(0)}$.
    Then, for any $\gamma_1,\ldots,\gamma_n \in G_x$,
    $[\psi(\gamma_i)+\psi(\gamma_j)-\psi(\gamma_i\gamma_j^{-1})]_{1\leq i, j \leq n}$
    is a positive-definite matrix.
\end{lemma}
\begin{proof}
Fix $\gamma_1,\ldots,\gamma_n \in G_x$, and let $\gamma_0 =x$.
    Take any $\lambda_1, \ldots, \lambda_n \in \bR$,
    and set $\lambda_0=-\sum_{j=1}^{n}\lambda_j$. 
    Thus $\sum_{i=0}^{n} \lambda_i=0$.
    Applying condition (iii) of Defintion~\ref{ndcf} to $\gamma_0, \gamma_1, \ldots, \gamma_n$, we have that $\sum_{j,k=0}^{n}\lambda_j \lambda_k \psi(\gamma_j\gamma_k^{-1}) \leq 0$.

    It follows that,
    \begin{align*}
        0 &\geq \sum_{j,k=0}^{n}\lambda_j \lambda_k \psi(\gamma_j\gamma_k^{-1}) \\
        &=\sum_{j,k=1}^{n}\lambda_j\lambda_k \psi(\gamma_j\gamma_k^{-1})+\sum_{j=1}^{n}\lambda_j \lambda_0 \psi(\gamma_j\gamma_0^{-1})+ \sum_{k=1}^{n} \lambda_0 \lambda_k \psi(x\gamma_k^{-1})+ \lambda_0^{2} \psi(x)\\
        &=\sum_{j,k=1}^{n}\lambda_j\lambda_k \psi(\gamma_j\gamma_k^{-1})+\sum_{j=1}^{n}\lambda_j \lambda_0 \psi(\gamma_jx)+ \sum_{k=1}^{n} \lambda_0 \lambda_k \psi(x\gamma_k^{-1})\\
        &= - \sum_{j, k=1}^{n} \lambda_j \lambda_k (\psi(\gamma_jx)+\psi(\gamma_kx)-\psi(\gamma_j\gamma_k^{-1})) \qquad (\text{since }\lambda_0 = - \sum_{k=1}^n\lambda_k)\\
        &= - \sum_{j, k=1}^{n} \lambda_j \lambda_k (\psi(\gamma_j)+\psi(\gamma_k)-\psi(\gamma_j\gamma_k^{-1})),
    \end{align*}
    which proves the result.
\end{proof}

\begin{theorem}\label{thm: neg to length}
    If $\psi\colon G \rightarrow [0, \infty)$ is a locally proper negative type function on a Hausdorff \'{e}tale groupoid $G$, then $L:=\sqrt{\psi}$ is a length function on the groupoid $G$. 
\end{theorem}
\begin{proof} 
Take any $x\in \GG$ and $\gamma, \eta \in G_x$.
    The matrix
    $$
    A_x:=\begin{bmatrix}
2 \psi(\gamma) & \psi(\gamma)+\psi(\eta)-\psi(\gamma \eta^{-1})\\
\psi(\gamma)+\psi(\eta)-\psi(\gamma \eta^{-1}) & 2 \psi(\eta) & 
\end{bmatrix}
$$
is positive-definite by Lemma~\ref{ndl}. 
Thus $\det (A_x) \geq 0$ and hence
\begin{align*}
    [\psi(\gamma)+\psi(\eta)-\psi(\gamma\eta^{-1})]^{2} \leq 4 |\psi(\gamma)||\psi(\eta)|.
\end{align*}
Taking the  square root of the above inequality, we get
\begin{align*}
    [\psi(\gamma)+\psi(\eta)-\psi(\gamma \eta^{-1})] \leq 2 \sqrt{|\psi(\gamma)|}\sqrt{|\psi(\eta)|},
\end{align*}
and so
\begin{align*}
  |\psi(\gamma \eta^{-1})|&=|(\psi(\gamma \eta^{-1})-\psi(\gamma)-\psi(\eta)+ \psi(\gamma)+\psi(\eta)|\\
  &\leq |\psi(\gamma)+\psi(\eta)|+ |(\psi(\gamma \eta^{-1})-\psi(\gamma)-\psi(\eta))|\\
  & \leq |\psi(\gamma)|+|\psi(\eta)| +2\sqrt{|\psi(\gamma)|}\sqrt{|\psi(\eta)|} \\
  &=(\sqrt{|\psi(\gamma)|} + \sqrt{|\psi(\eta)|})^{2}\,.
\end{align*}
Consequently, by taking square root on the both sides of the above inequality and using symmetric property of $\psi$, one gets
\begin{align*}
  \sqrt{\psi(\gamma \eta^{-1})} &\leq \sqrt{\psi(\gamma)} + \sqrt{\psi(\eta)}
    = \sqrt{\psi(\gamma)} + \sqrt{\psi(\eta^{-1})}.
\end{align*}
Hence, $L = \sqrt{\psi}$ is a sub-additive function.
As $\psi$ is normalized and symmetric, $L$ is normalized and symmetric.
Consequently, $L$ is a length function on $G$.
\end{proof}

When $G$ is a group with a locally proper negative type function $\psi$, the length function $L= \sqrt{\psi}$ is called a \emph{Haagerup length function} in \cite{BedCon2009}.
The following is a groupoid analogue of \cite[Theorem~5.9]{BedCon2009}.

\begin{theorem}\label{thm: main}
    Let $G$ be locally compact Hausdorff \'{e}tale groupoid, and let $\Sigma \rightarrow G$ be a twist.
    Suppose there is a non-zero locally proper negative type function $\psi$ on $G$, and that $G$ has the rapid decay property with respect to the length function $L = \sqrt{\psi}.$
    Take any $f \in C_r^*(\Sigma; G)$ and let $U$ be the open support of $f$.
    Then $f \in \overline{C_c(\Sigma|_U;U)}^{\|\cdot\|_r}$.

    In particular, the conclusion holds if $G$ is second-countable and satisfies the Haagerup property and has the rapid decay property with respect to the induced length function.
\end{theorem}
\begin{proof}
    The proof is similar to the proof of Theorem~\ref{thm: negative RDP}.
    We outline the details.
    Define a sequence of functions $k_n(\gamma)=e^{-\frac{\psi(\gamma)}{n}} \in C(G)$ for all $n \in \bN$. 
    We again have that the sequence $(k_n)_n$ are positive-definite, bounded by $1$ in $\sup$-norm, and converge uniformly to $1$ on compact subsets of $G$.
    Further, as $e^{-x}(1+\sqrt{x})^{s} \rightarrow 0$ as $x \rightarrow \infty$ for any positive $s \in \mathbb{R}$, we have $|k_n(\gamma)|(1+L(\gamma))^{p}$
    is bounded for each $p$, where $L = \sqrt{\psi}$.
    The rest of the proof now follows the same argument as that of Theorem~\ref{thm: negative RDP}.
\end{proof}


\end{document}